\documentclass[a4paper]{article}
\usepackage[utf8]{inputenc}
\usepackage[T1]{fontenc} 
\usepackage{fancyhdr}
\usepackage[colorlinks,citecolor=blue]{hyperref}

\usepackage[style = alphabetic, sorting = nyt]{biblatex}
\usepackage[title]{appendix}
\usepackage{amsfonts,amsmath,amssymb,amsthm}
\usepackage{enumerate}
\usepackage{bbm} 
\usepackage{verbatim}
\usepackage{authblk}
\usepackage{times}
\usepackage{mdframed}
\usepackage{lipsum}
\usepackage{soul}

\usepackage{graphicx,float}
\usepackage{tikz,tkz-euclide,pgfplots}
\usetikzlibrary{calc,patterns}
\usetikzlibrary{arrows,shapes,positioning}

\usepackage{epigraph} 
\usepackage{csquotes}
\SetBlockThreshold{2}

\DeclareMathOperator{\dens}{dens}

\DeclareMathOperator{\Red}{Red}
\DeclareMathOperator{\Edge}{Edge}

\newcommand{\moinsun}{^{-1}}

\newcommand{\tendversl}{\xrightarrow[\ell\to\infty]{}}

\newcommand\flr[1]{\lfloor #1 \rfloor}

\def\Pr{\mathbf{Pr}}

\newcommand\Cond[2]{\left( #1 \;\middle|\; #2 \right)}
\def\tilde{\widetilde}

\newtheorem{thm}{Theorem}[section]

\newtheorem{defi}[thm]{Definition}
\newtheorem{lem}[thm]{Lemma}
\newtheorem{prop}[thm]{Proposition}
\newtheorem{cor}[thm]{Corollary}

\theoremstyle{remark}
\newtheorem{rem}[thm]{Remark}

\title{Phase transition for the existence of van Kampen 2-complexes in random groups}
\author{\textsc{Tsung-Hsuan Tsai}}
\affil{Institut de Recherche Mathématique Avancée\\
Strasbourg University\\
\small{\textit{E-mail:} tsai71517@gmail.com}}
\date{}

\begin{document}
\maketitle

\begin{abstract}
    Gromov showed that \cite{Gro93} with high probability, every bounded and reduced van Kampen diagram $D$ of a random group at density $d$ satisfies the isoperimetric inequality $|\partial D|\geq (1-2d-s)|D|\ell$. In this article, we adapt Gruber-Mackay's prove \cite{GM18} for random triangular groups, showing a non-reduced 2-complex version of this inequality.\\
    
    Moreover, for any 2-complex $Y$ of a given geometric form, we exhibit a phase transition: we give explicitly a critical density $d_c$ depending only on $Y$ such that, in a random group at density $d$, if $d<d_c$ then there is no reduced van Kampen 2-complex of the form $Y$; while if $d>d_c$ then there exists reduced van Kampen 2-complexes of the form $Y$.\\
    
    As an application, we show a phase transition for the $C(p)$ small-cancellation condition: for a random group at density $d$, if $d<1/(p+1)$ then it satisfies $C(p)$; while if $d>1/(p+1)$ then it does not satisfy $C(p)$.
\end{abstract}
\tableofcontents

\section{Introduction}

\paragraph{Random groups.}
The first mention of random group presentations is the density model by M. Gromov in \cite[9.B]{Gro93}. Formally, a random group is a random variable with values in a given set of groups, often constructed by group presentations with a fixed set of generators and a random set of relators. The goal is to study the asymptotic behaviors of a sequence of random groups $(G_\ell)$ when the maximal relator lengths $\ell$ goes to infinity. We say that $G_\ell$ satisfies some property $Q_\ell$ \textit{asymptotically almost surely} (a.a.s.) if the probability that $G_\ell$ satisfies $Q_\ell$ converges to $1$ as $\ell$ goes to infinity.

Let us consider the \textit{permutation invariant density model} of random groups introduced by Gromov in \cite[p. 272]{Gro93} and developed in \cite{Tsa21}. Fix the set of generators $X_m=\{x_1,\dots,x_m\}$ with $m\geq 2$ for group presentations. Let $B_\ell$ be the set of cyclically reduced words of $X_m^\pm$ of length at most $\ell$. We shall construct random groups by \textit{densable} and \textit{permutation invariant} random subsets of $B_\ell$.

\begin{defi}[{\cite[p.272]{Gro93}}, {\cite[Definition 1.5 and Definition 2.5]{Tsa21}}] A sequence of random subsets $(R_\ell)$ of the sequence of sets $(B_\ell)$ is called \textit{densable with density} $d\in \{-\infty\} \cup [0,1]$ if the sequence of random variables $\dens_{B_\ell}(R_\ell) := \log_{|B_\ell|}(|R_\ell|)$ converges in probability to the constant $d$.

The sequence $(R_\ell)$ is called \textit{permutation invariant} if $R_\ell$ is a permutation measure-invariant random subset of $B_\ell$.
\end{defi}

Many natural models of random subsets are densable and permutation invariant. For example, the uniform distribution on all subsets of cardinality $\flr{B_\ell}^d$ considered in \cite{Oll04},\cite{Oll05} and \cite{Oll07}, or the Bernoulli sampling of parameter $|B_\ell|^{d-1}$ considered in \cite{ALS15} for random triangular groups.

\begin{defi}[{\cite[p.273]{Gro93}}, {\cite[Definition 4.1]{Tsa21}}] A sequence of random groups $(G_\ell(m,d))$ with $m$ generators at density $d$ is defined by
\[G_\ell(m,d) = \langle X_m|R_\ell\rangle\] 
where $(R_\ell)$ is a densable sequence of permutation invariant random subsets of $(B_\ell)$ with density $d$.
\end{defi}

For detailed surveys on random groups, we refer the reader to \cite{Ghy04} by E. Ghys, \cite{Oll05} by Y. Ollivier, \cite{KS08} by I. Kapovich and P. Schupp, and \cite{BNW20} by F. Bassino, C. Nicaud and P. Weil.

\paragraph{Isoperimetric inequalities.} 

In order to prove the hyperbolicity of a random group at density $d<1/2$, Gromov showed in \cite[9.B]{Gro93} that a.a.s. \textit{reduced} local van Kampen diagrams of $G_\ell(m,d)$ satisfy an isoperimetric inequality depending on the density $d$.

\begin{thm}[{\cite[p.274]{Gro93}}, {\cite[Chapter 2]{Oll04}}]\label{inequality for diagrams} Let $(G_\ell(m,d))$ be a sequence of random groups with $m\geq2$ generators at density $d$. For any $\varepsilon>0$ and $K>0$, a.a.s. every \textbf{reduced} van Kampen diagram $D$ of $G_\ell(m,d)$ with $|D|\leq K$ satisfies the isoperimetric inequality
\[|\partial D|\geq (1-2d-\varepsilon)|D|\ell.\]
\end{thm}
Ollivier's proof in \cite{Oll04} can achieve a slightly stronger\footnote{Note that every van Kampen diagram composed of relators of lengths at most $\ell$ satisfies $2|D^{(1)}|-|\partial D|\leq |D|\ell$, so the given inequality implies the isoperimetric inequality.} inequality \[|D^{(1)}| \geq (1-d-\frac{\varepsilon}{2})|D|\ell.\]

One may expect such an inequality to hold for \textit{every} reduced van Kampen 2-complex $Y$ with $|Y|\leq K$. In \cite[Section 2]{GM18}, D. Gruber and J. Mackay showed that in the triangular model of random groups\footnote{A model that the relator length $\ell = 3$ is fixed, and we are interested in asymptotic behaviors when the number of generators $m$ goes to infinity.}, the above inequality holds for every \textit{non-reduced van Kampen 2-complexes} $Y$ with $|Y|\leq K$ if the \textit{reduction degree} (Definition \ref{reduction degree}) $\Red(Y)$ is added in the left hand side of the inequality.\\

However, the result fails in the regular Gromov density model: the condition $|Y|\leq K$ is not enough (see Remark \ref{why bounded complexity}). In Section 2 of this paper, we introduce the notion of \textit{complexity} (Definition \ref{complexity}) to adapt Gruber-Mackay's inequality in the Gromov density model, establishing a \textit{non-reduced van Kampen 2-complex} version of Theorem \ref{inequality for diagrams}.

\begin{thm} \label{inequality for 2-complex} Let $(G_\ell(m,d))$ be a sequence of random groups with $m\geq2$ generators at density $d$. Let $\varepsilon>0$, $K>0$. For any $d<1/2$, a.a.s. every van Kampen 2-complex $Y$ of complexity $K$ of $G_\ell(m,d)$ satisfies 
\[|Y^{(1)}|+\Red(Y)\geq (1-d-\varepsilon)|Y|\ell.\]
\end{thm}

\paragraph{Phase transition for the existence of van Kampen 2-complexes.} We are now interested in the converse of Theorem \ref{inequality for 2-complex}: Given a 2-complex $Y$ satisfying the inequality of Theorem \ref{inequality for 2-complex}, is it true that a.a.s. there exists a \textit{reduced} van Kampen 2-complex of $G_\ell(m,d)$ whose underlying 2-complex is $Y$?

A 2-complex $Y$ is said to be \textit{fillable} by a group presentation $G = \langle X |R \rangle$ (or by the set of relators $R$) if there exists a \textit{reduced} van Kampen 2-complex of $G$ whose underlying 2-complex is $Y$. An edge of a 2-complex is called \textit{isolated} if it is not adjacent to any face. Since isolated edges do not affect fillability, we will only consider finite 2-complexes without isolated edges in the following.\\

To better formulate the problem, we consider a sequence of 2-complexes $(Y_\ell)$ and introduce the notion of \textit{geometric form} of 2-complexes $(Y,\lambda)$ (Definition \ref{geometric form}), together with its density $\dens Y$ and its \textit{critical density} $\dens_c Y$ (Definition \ref{def density of Y}). The main result of this article is the phase transition at density $1- \dens_c(Y)$, for the fillability of the 2-complex $Y_\ell$.

\begin{thm}\label{existence of 2-complexes}
Let $(G_\ell(m,d))$ be a sequence of random groups with $m\geq2$ generators at density $d$. Let $(Y_\ell)$ be a sequence of 2-complexes with some geometric form $(Y,\lambda)$.

\begin{enumerate}[(i)]
    \item If $d < 1 - \dens_c Y$, then a.a.s. $Y_\ell$ is not fillable by $G_\ell(m,d)$.
    \item If $d > 1 - \dens_c Y$ and $Y_\ell$ is fillable by $B_\ell$, then a.a.s. $Y_\ell$ is fillable by $G_\ell(m,d)$.
\end{enumerate}
\end{thm}\quad

In Section \ref{section existence of 2-complexes}, we prove Theorem \ref{existence of 2-complexes} using the multidimensional intersection formula for random subsets (Theorem \ref{multidim intersection}, \cite[Theorem 3.7]{Tsa21}), which generalizes the proof for the $C'(\lambda)$ phase transition in \cite[Theorem 1.4]{Tsa21}. We will see in Remark \ref{theorem implies corollary} that the second assertion of the theorem is equivalent to the following corollary.

\begin{cor}\label{cor of theorem} Let $(G_\ell(m,d))$ be a sequence of random groups with $m\geq2$ generators at density $d$. Let $s>0$ and $K>0$. Let $(Y_\ell)$ be a sequence of 2-complexes of the same geometric form such that $Y_\ell$ is fillable by $B_\ell$. If every sub-2-complex $Z_\ell$ of $Y_\ell$ satisfies
    \[|Z^{(1)}_\ell|\geq (1-d+s)|Z_\ell|\ell,\]
then a.a.s. $Y_\ell$ is fillable by $G_\ell(m,d)$.
\end{cor}

Note that we need $Y_\ell$ to have at least one filling by the set of all possible relators $B_\ell$. It is automatically satisfied for planar and simply connected 2-complexes. In addition, if every face boundary length of $Y_\ell$ is exactly $\ell$, then the given inequality is equivalent to an isoperimetric inequality similar the inequality of Theorem \ref{inequality for diagrams}. Hence the following corollary.

\begin{cor}\label{existence of diagrams} Let $(G_\ell(m,d))$ be a sequence of random groups with $m\geq2$ generators at density $d$. Let $s>0$ and $K>0$. Let $(D_\ell)$ be a sequence of finite planar 2-complexes of the same geometric form such that every face boundary length of $D_\ell$ is exactly $\ell$. If every sub-2-complex $D_\ell'$ of $D_\ell$ satisfies
    \[|\partial D'_\ell|\geq (1-2d+s)|D'_\ell|\ell,\]
then a.a.s. $D_\ell$ is fillable by $G_\ell(m,d)$.
\end{cor}\quad

It is mentioned in \cite[Proposition 1.8]{OW11} that when $d<1/(p+1)$, a.a.s. a random group at density $d$ has the $C(p)$ small cancellation condition. As an application of Theorem \ref{existence of 2-complexes}, we show that there is a phase transition: if $d>1/(p+1)$, then a.a.s. a random group at density $d$ \textit{does not} have $C(p)$ (see Proposition \ref{c of p}).

\paragraph{Acknowledgements.} The content of this article is completed during my PhD thesis \cite[Chapter 4]{Tsa22} at the University of Strasbourg. I would like to thank my thesis advisor, Thomas Delzant, for his guidance and interesting discussions on the subject.

\section{Isoperimetric inequality for van Kampen 2-complexes}\label{section vk 2-complexes}

We shall prove Theorem \ref{inequality for 2-complex} in this section.

\paragraph{Van Kampen 2-complexes.}
We consider oriented combinatorial 2-complexes and van Kampen diagrams as in \cite{LS77}.

A 2-complex is a triplet $Y = (V,E,F)$ where $V$ is the set of vertices, $E$ is the set of oriented edges and $F$ is the set of oriented faces. Every edge $e\in E$ has a starting point $\alpha(e)\in V$, an ending point $\omega(e)\in V$ and an inverse edge $e\moinsun\in E$, satisfying $\alpha(e\moinsun)=\omega(e)$, $\omega(e\moinsun) = \alpha(e)$ and $(e\moinsun)\moinsun = e$. A \textit{geometric edge} is a pair of inverse edges $\overline e$, denoted by $\overline e$. Every face $f\in F$ has a boundary $\partial f$, which is a cyclically reduced loop of the under lying combinatorial oriented graph $Y^{(1)} = (V,E)$, and an inverse face $f\moinsun\in F$ satisfying $\partial(f\moinsun) = (\partial f)\moinsun$ and $(f\moinsun)\moinsun = f$. A \textit{geometric face} is a pair of inverse faces $\{f,f\moinsun\}$, denoted by $\overline f$. We denote $|Y^{(1)}|$ the number of geometric edges and $|Y|$ the number of geometric faces.

A \textit{van Kampen 2-complex} with respect to a group presentation $G = \langle X | R\rangle$ is a 2-complex $Y = (V,E,F)$ with two compatible labeling functions: labels on edges by generators $\varphi_1 : E\to X^\pm$, and labels on faces by relators $\varphi_2 : F\to R^\pm$. Compatible means that $(V,E,\varphi_1)$ is a labeled graph, $\varphi_2(f\moinsun) = \varphi_2(f)\moinsun$ and $\varphi_1(\partial f) = \varphi_2(f)$. The data of the labels $\varphi_1,\varphi_2$ on $Y$ is equivalently given by a combinatorial map $Y\to K(X,R)$ where $K(X,R)$ is the standard 2-complex with respect to the group presentation  $G = \langle X | R\rangle$ (with one vertex, an edge for each generator and its inverse, and a face for each relator and its inverse). We denote briefly $Y = (V,E,F,\varphi_1,\varphi_2)$.

A \textit{van Kampen diagram} $D$ is a finite, planar (embedded in the plan) and simply connected van Kampen 2-complex. Its boundary length $|\partial D|$ is the length of a boundary path, passing once by every edge adjacent to one face and twice by every edge adjacent to zero face.

A pair of faces in a van Kampen 2-complex is called \textit{reducible} if they have the same relator label and there is a common edge on their boundaries at the same position. A van Kampen 2-complex is called \textit{reduced} if there is no reducible pair of faces.

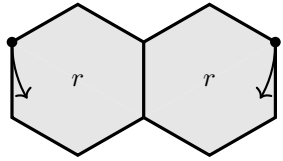
\begin{figure}[h]
    \centering
    \begin{tikzpicture}
        \fill[gray!20] (0,0) -- (0,1) --+ (30:1) -- (30:2);
        \fill[gray!20] (0,0) -- (-30:1) --+ (30:1) -- (30:2);
        \fill[gray!20] (0,0) -- (0,1) --+ (150:1) -- (150:2);
        \fill[gray!20] (0,0) -- (210:1) --+ (150:1) -- (150:2);
        \draw[very thick] (0,0) -- (0,1) --+ (30:1) -- (30:2);
        \draw[very thick] (0,0) -- (-30:1) --+ (30:1) -- (30:2);
        \draw[very thick] (0,1) --+ (150:1) -- (150:2);
        \draw[very thick] (0,0) -- (210:1) --+ (150:1) -- (150:2);
        \draw [thick, ->] (30:2) arc (0:-30:1.5);
        \draw [thick, ->] (150:2) arc (180:210:1.5);
        \fill (30:2) circle (0.07);
        \fill (150:2) circle (0.07);
        \node at (30:1) {$r$};
        \node at (150:1) {$r$};
    \end{tikzpicture}
    \caption{A reducible pair of faces.}
    \label{reducible faces}
\end{figure}

\subsection{Reduction degree and Complexity}
Let us define the \textit{reduction degree} of a van Kampen 2-complex and the \textit{complexity} of a 2-complex.

The \textit{reduction degree} of a non-reduced van Kampen diagram $Y = (V,E,F,\varphi_1,\varphi_2)$ with respect to a group presentation $\langle X|R\rangle$ is the total number of geometric edges causing reducible pair of faces, counted with \textit{multiplicity}: for any edge $e\in E$, any relator $r\in R$ and any integer $j$, we count the number of faces $f\in F$ labeled by $r$ and having $e$ as the $j$-th boundary edge. If this number is $k$, we add $(k-1)^+$ to the reduction degree where $(\cdot)^+$ is the positive part function.

\begin{defi}[Reduction degree, {\cite[Definition 2.5]{GM18}}]\label{reduction degree} Let $Y = (V,E,F,\varphi_1,\varphi_2)$  be a van Kampen 2-complex of a group presentation $G = \langle X|R\rangle$. Let $\ell$ be the maximal boundary length of faces of $Y$. The reduction degree of $Y$ is
\[\Red(Y) = \sum_{e\in E}\sum_{r\in R}\sum_{1\leq j\leq \ell}\Big(\big|\{f\in F \,|\, \varphi_2(f) = r, e \textup{ is the $j$-th edge of } \partial f\}\big| - 1 \Big)^+.\]
\end{defi}

It is not hard to see that a van Kampen 2-complex $Y$ is reduced if and only if $\Red(Y) = 0$.\\

A \textit{maximal arc} of a 2-complex is a reduced combinatorial path passing only by vertices of degree $2$ whose endpoints are not of degree $2$. The \textit{complexity} of a 2-complex encodes the number of maximal arcs with the number of faces.

\begin{defi}[Complexity of a 2-complex]\label{complexity}
Let $Y$ be a 2-complex. Let $K>0$. We say that $Y$ is of complexity $K$ if the following three conditions hold:
    \begin{itemize}
        \item $|Y|\leq K$.
        \item The number of maximal arcs of $Y$ is bounded by $K$.
        \item For any face $f$ of $Y$, the boundary path $\partial f$ is divided into at most $K$ maximal arcs.
    \end{itemize}
\end{defi}

Note that if $D$ is a planar and simply connected 2-complex with $|D|\leq K$, then the complexity of $D$ is $6K$. In fact, as the rank of its underlying graph is $K$, the number of its maximal arcs is at most $3K$, and every boundary path is divided into at most $6K$ maximal arcs (an arc may be used twice).

\begin{lem}\label{number of abstract 2-complexes}
Let $K>0$. For $\ell$ large enough, the number of 2-complexes of complexity $K$ with face boundary lengths at most $\ell$ is bounded by $\ell^{3K}$.
\end{lem}
\begin{proof}
As there are at most $K$ maximal arcs, the number of vertices of valency greater than $3$ is at most $2K/3\leq K$. There are at most $K^2$ ways to draw an arc connecting two of these vertices, every arc is with length at most $\ell$. The number of underlying graphs is then at most $(K^2)^K\ell^K$.

To attach $K$ faces on a graph, we choose $K$ loops passing by at most $K$ arcs, there are at most $(K^2)^{K^2}$ choices. There are at most $(2\ell)^{K}$ ways to choose a starting point and an orientation for every face. The number of such 2-complexes is hence bounded by 
\[(K^2)^K\ell^K \times (K^2)^{K^2} \times (2\ell)^{K},\]
which is smaller than $\ell^{3K}$ if $\ell$ is large enough.
\end{proof}

\begin{rem}\label{why bounded complexity}
While the number of 2-complexes with a \textit{bounded complexity} grows polynomially with the maximal face boundary length $\ell$, it is not the case for 2-complexes with a \textit{bounded number of faces}. Hence the related argument in \hyperref[proof main theorem]{Proof of Theorem \ref{inequality for 2-complex}} does not work. Actually, there exists van Kampen 2-complexes that contradicts the inequality of Theorem \ref{inequality for 2-complex}.

For example, in \cite{CW15}, D. Calegari and A. Walker proved that at any density $d<1/2$, there exists a number $K$ depending only on $d$ such that a.a.s. there is a reduced van Kampen 2-complex $Y$ homeomorphic to a surface of genus $O(\ell)$ in $G_\ell(m,d)$ with at most $K$ faces. Since every edge is adjacent to two faces in a surface, we have $|Y^{(1)}|\leq \frac{1}{2}|Y|\ell$, while we expect that $|Y^{(1)}|\geq \left(1-d-\frac{s}{2}\right) >\frac{1}{2}|Y|\ell$.
\end{rem}

\subsection{Abstract van Kampen 2-complexes}
Let $(G_\ell(m,d))$ be a sequence of random groups at density $d$, defined by $G_\ell(m,d) = \langle x_1,\dots,x_m | R_\ell\rangle$. Recall that $B_\ell$ is the set of all cyclically reduced words of length at most $\ell$ and $|B_
ell| = (2m-1)^{\ell+O(1)}$. Let $0<\varepsilon<1-d$. Since $\log_{|B_\ell|}|R_\ell|$ converges in probability to the constant $d$, the probability event
\[Q_\ell:=\left\{(2m-1)^{(d-\frac{\varepsilon}{4}\ell)} \leq |R_\ell| \leq (2m-1)^{(d+\frac{\varepsilon}{4}\ell)}\right\}\]
is a.a.s. true (cf. \cite{Tsa21} Proposition 1.8).\\

If we consider the Bernoulli density model that the events $\{r\in R_\ell\}$ through $r\in B_\ell$ are independent of the same probability $(2m-1)^{(d-1)\ell}$, it is obvious that we have $\Pr(r_1,\dots,r_k\in R_\ell) = (2m-1)^{k(d-1)\ell}$ for distinct $r_1, \dots, r_k$ in $B_\ell$. In the permutation invariant density model, we have the following corresponding proposition, which is a variant of \cite{Tsa21} Lemma 3.10. 

\begin{prop} \label{k relators in R} Let $r_1, \dots, r_k$ be pairwise different relators in $B_\ell$. We have \[\Pr\Cond{r_1,\dots,r_k\in R_\ell}{Q_\ell}\leq (2m-1)^{k(d-1+\frac{\varepsilon}{2})\ell}.\]\\[-3em]
\qed
\end{prop}\quad

Abstract van Kampen 2-complexes, as abstract van Kampen diagrams introduced by Ollivier in \cite{Oll04}, is a structure between 2-complexes and van Kampen 2-complexes that helps us solve 2-complex problems in random groups. Recall that since isolated edges do not affect fillability, we will only consider finite 2-complexes without isolated edges.

\begin{defi}[Abstract van Kampen 2-complex] An abstract van Kampen 2-complex $\tilde Y$ is a 2-complex $(V,E,F)$ with a labeling function on faces by integer numbers and their inverses $\tilde\varphi_2: F\to \{1,1^-,2,2^-,\dots,k,k^-\}$ such that $\tilde\varphi_2(f\moinsun) = \tilde\varphi_2(f)^-$. 
We denote simply $\tilde Y = (V,E,F,\tilde\varphi_2)$.
\end{defi}

By convention $(i^-)^- = i$. The integers $\{1,\dots, k\}$ are called abstract relators. Similar to a van Kampen diagram, a pair of faces $f,f'\in F$ is \textit{reducible} if they are labeled by the same abstract relator, and they share an edge at the same position of their boundaries. An abstract diagram is called \textit{reduced} if there is no reducible pair of faces. Let $\ell$ be the maximal boundary length of faces. The \textit{reduction degree} of $\tilde Y$ 2-complex can be similarly defined as
\[\Red(\tilde Y) = \sum_{e\in E}\sum_{1\leq i\leq k}\sum_{1\leq j\leq \ell}\Big(\big|\{f\in F \,|\, \tilde\varphi_2(f) = i, e \textup{ is the $j$-th edge of } \partial f\}\big| - 1 \Big)^+.\]

We say that an abstract van Kampen 2-complex with $k$ abstract relators $\tilde Y = (V,E,F,\tilde\varphi_2)$ is \textit{fillable} by a group presentation $G = \langle X|R\rangle$ (or by a set of relators $R$) if there exists $k$ \textit{different} relators $r_1,\dots,r_k\in R$ such that the construction $\varphi_2(f) := r_{\tilde\varphi_2(f)}$ gives a van Kampen 2-complex $Y = (V,E,F,\varphi_1,\varphi_2)$\footnote{Note that the edge labeling $\varphi_1$ is determined by the face labeling $\varphi_2$ as there is no isolated edges.} of $G$. The $k$-tuple of relators $(r_1,\dots,r_k)$, or the van Kampen 2-complex $Y$, is called a \textit{filling} of $\tilde Y$. As we picked different relators for different abstract relators, if $Y$ is a filling of $\tilde Y$, then $\Red(Y) = \Red(\tilde Y)$, and $\tilde Y$ is reduced if and only if $Y$ is reduced.

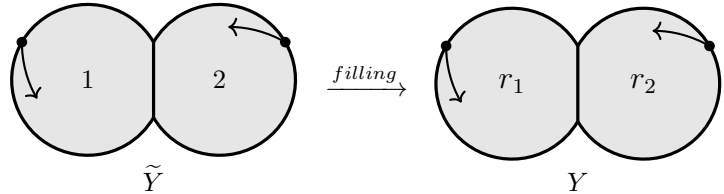
\begin{figure}[h]
\centering
    \begin{tikzpicture}
            \fill[gray!20] (30:1) circle (1) ;
            \fill[gray!20] (150:1) circle (1) ;
            \draw[very thick] (0,0) -- (0,1);
            \draw[very thick] (0,0) arc (-150:150:1);
            \draw[very thick] (0,1) arc (30:330:1);
            \node at (30:1) {$2$};
            \node at (150:1) {$1$};
            \draw [thick, ->] (30:2) arc (60:90:1.5);
            \draw [thick, ->] (150:2) arc (180:210:1.5);
            \fill (30:2) circle (0.07);
            \fill (150:2) circle (0.07);
            \node at (2.8,0.5) {$\xrightarrow{filling}$};
            \node at (0,-0.8) {$\tilde Y$};
    \end{tikzpicture}
    \begin{tikzpicture}
            \fill[gray!20] (30:1) circle (1) ;
            \fill[gray!20] (150:1) circle (1) ;
            \draw[very thick] (0,0) -- (0,1);
            \draw[very thick] (0,0) arc (-150:150:1);
            \draw[very thick] (0,1) arc (30:330:1);
            \node at (30:1) {\large{$r_2$}};
            \node at (150:1) {\large{$r_1$}};
            \draw [thick, ->] (30:2) arc (60:90:1.5);
            \draw [thick, ->] (150:2) arc (180:210:1.5);
            \fill (30:2) circle (0.07);
            \fill (150:2) circle (0.07);
            \node at (0,-0.8) {$Y$};
    \end{tikzpicture}
    \caption{Filling an abstract van Kampen 2-complex.}
    \label{filling an abstract diagram}
\end{figure}

Denote $\ell_i$ the length of the abstract relator $i$ for $1\leq i \leq k$. Let $\ell = \max\{\ell_1, \dots, \ell_k\}$ be the maximal boundary length of faces. The pairs of integers $(i,1), \dots, (i,\ell_i)$ are called \textit{abstract letters} of $i$. The set of abstract letters of $\tilde Y$ is then a subset of the product set $\{1,\dots,k\} \times \{1,\dots,\ell\}$. The geometric edges of $\tilde Y$ are decorated by abstract letters and directions: Let $f\in F$ labeled by $i$ and let $e\in E$ at the $j$-th position of $\partial f$. The geometric edge $\overline e$ is decorated, on the side of $\overline f$, by an arrow indicating the direction of $e$ and the abstract letter $(i,j)$. The number of decorations on a geometric edge is the number of its adjacent faces with multiplicity (an edge may be attached twice by the same face).

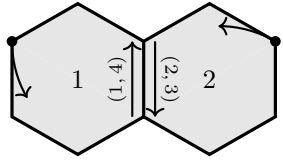
\begin{figure}[h]
    \centering
\begin{tikzpicture}
        \fill[gray!20] (0,0) -- (0,1) --+ (30:1) -- (30:2);
        \fill[gray!20] (0,0) -- (-30:1) --+ (30:1) -- (30:2);
        \fill[gray!20] (0,0) -- (0,1) --+ (150:1) -- (150:2);
        \fill[gray!20] (0,0) -- (210:1) --+ (150:1) -- (150:2);
        \draw[very thick] (0,0) -- (0,1) --+ (30:1) -- (30:2);
        \draw[very thick] (0,0) -- (-30:1) --+ (30:1) -- (30:2);
        \draw[very thick] (0,1) --+ (150:1) -- (150:2);
        \draw[very thick] (0,0) -- (210:1) --+ (150:1) -- (150:2);
        \draw [thick, ->] (30:2) arc (60:90:1.5);
        \draw [thick, ->] (150:2) arc (180:210:1.5);
        \fill (30:2) circle (0.07);
        \fill (150:2) circle (0.07);
        \draw [thick, ->] (0.15,1) -- (0.15,0);
        \draw [thick, ->] (-0.15,0) -- (-0.15,1);
        \node at (30:1) {$2$};
        \node at (150:1) {$1$};
        \node[rotate = -90] at (0.35,0.5) {\scriptsize{$(2,3)$}};
        \node[rotate = 90] at (-0.35,0.5) {\scriptsize{$(1,4)$}};
\end{tikzpicture}
    \caption{A geometric edge decorated by two abstract letters.}
    \label{diagram decoration}
\end{figure}

\begin{defi}[free-to-fill]\label{free to fill} An abstract letter $(i,j)$ of $\tilde D$ is \textbf{free-to-fill} if, for any edge $\overline e$ decorated by $(i,j)$, it is the minimal decoration on $\overline e$.
\end{defi}

Denote $\alpha_i$ the number of faces labeled by the abstract relator $i$ and $\eta_i$ the number of free-to-fill edges of $i$. We have the following estimation.

\begin{lem}\label{alphai etai 2} Let $\tilde Y = (V,E,F,\tilde \varphi_2)$ be an abstract van Kampen 2-complex with $k$ abstract relators.  Then
\[\sum_{i=1}^k\alpha_i\eta_i\leq |\tilde Y^{(1)}|+\Red(\tilde Y).\]
\end{lem}
\begin{proof} Denote $\overline E$ the set of geometric edges and $\overline F$ the set of geometric faces. For any geometric edge $\overline e$, an adjacent face $\overline f$ from which the decoration is minimal is called a \textit{preferred face} of $\overline e$. For any face $\overline f$, let $\overline E_{\overline f}$ be the set of geometric edges $\overline e$ on its boundary such that $\overline f$ is a preferred face of $\overline e$. Note that an edge will never be counted twice as the decorations given by one face are all different. According to Definition \ref{free to fill}, for any face $f$ with $\tilde\varphi_2(f) = i$, we have $\eta_i\leq |\overline E_{\overline f}|$. Hence,
\[\sum_{i=1}^k\alpha_i\eta_i\leq \sum_{\overline f \in \overline F} |\overline E_{\overline f}|.\]

Denote $\Red(\overline e)$ the reduction degree caused by the edge $\overline e$. That is,
\[\Red(\overline e):= \sum_{1\leq i\leq k}\sum_{1\leq j\leq \ell}\Big(\big|\{f\in F \,|\, \tilde\varphi_2(f) = i,  \textup{ $e$ or $e\moinsun$ is the $j$-th edge of } \partial f\}\big| - 1 \Big)^+,\]
so that the number of preferred faces of $\overline e$ is bounded by $1+\Red(\overline e)$. Hence, \[\sum_{\overline f \in \overline F} |\overline E_{\overline f}| \leq \sum_{\overline e \in \overline E}\Big(1+\Red(\overline e)\Big) = |\tilde Y^{(1)}|+\Red(\tilde Y).\]
\end{proof}

\paragraph{Probability of filling.} We shall estimate the probability that an abstract van Kampen 2-complex $\tilde Y$ is fillable by a random group $G_\ell(m,d)$. This step is the key to prove Theorem \ref{inequality for 2-complex}. Recall that $Q_\ell:=\left\{(2m-1)^{(d-\frac{\varepsilon}{4}\ell)} \leq |R_\ell| \leq (2m-1)^{(d+\frac{\varepsilon}{4}\ell)}\right\}$ is an a.a.s. true probability event.

\begin{lem}\label{for sub 2-complexes} Let $\tilde Y$ be an abstract van Kampen 2-complex with $k$ abstract relators. We have
\[\Pr\Cond{\tilde Y \textup{ is fillable by } G_\ell(m,d)}{Q_\ell} \leq \left(\frac{2m}{2m-1}\right)^k (2m-1)^{\sum_{i=1}^k\left( \eta_i + \left(d-1+\frac{\varepsilon}{2}\right)\ell \right)}.\]
\end{lem} 
\begin{proof} Let us estimate the number of fillings of $\tilde Y$. For every free-to-fill abstract letter $(i,j)$, there are at most $2m$ ways to fill a generator if $j=1$, at most $(2m-1)$ ways to fill if $j\neq 1$ for avoiding reducible word. As there are $\eta_i$ free-to-fill abstract letters on the $i$-th abstract relator, there are at most $2m(2m-1)^{\eta_i-1}$ ways to fill it. So there are at most $\prod_{i=1}^k\left(2m (2m-1)^{\eta_i-1}\right)$ ways to fill $\tilde Y$.

Let $Y$ be a van Kampen 2-complex, which is a filling of $\tilde Y$. The 2-complex $Y$ is labeled by $k$ different relators in $B_\ell$, denoted $r_1,\dots, r_k$. By lemma \ref{k relators in R},
\begin{align*}
    \Pr\Cond{Y \textup{ is a 2-complex of } G_\ell(m,d)}{Q_\ell} & = \Pr\Cond{r_1,\dots,r_k\in R_\ell}{Q_\ell}\\
    & \leq (2m-1)^{k(d-1+\frac{\varepsilon}{2})\ell}.
\end{align*}

Hence
\begin{align*}
    \Pr\Cond{\tilde Y \textup{ is fillable by } G_\ell(m,d)}{Q_\ell} & \leq \sum_{Y\textup{ fills } \tilde Y}\Pr\Cond{Y \textup{ is a 2-complex of } G_\ell(m,d)}{Q_\ell} \\
    & \leq \prod_{i=1}^k\left(2m (2m-1)^{\eta_i-1}\right)(2m-1)^{k(d-1+\frac{\varepsilon}{2})\ell} \\
    & \leq \left(\frac{2m}{2m-1}\right)^k (2m-1)^{\sum_{i=1}^k\left( \eta_i + \left(d-1+\frac{\varepsilon}{2}\right)\ell \right)}.
\end{align*}
\end{proof}\quad

\begin{lem}\label{key lemma} Let $\tilde Y$ be an abstract van Kampen 2-complex with $k$ abstract relators. Suppose that $\tilde Y$ does not satisfy the inequality given in Theorem \ref{inequality for 2-complex}, i.e.
\[|\tilde Y^{(1)}| + \Red(\tilde Y) < (1-d-\varepsilon)|\tilde Y|\ell,\]
then

\[\Pr\Cond{\tilde Y \textup{ is fillable by } G_\ell(m,d)}{Q_\ell} \leq \left(\frac{2m}{2m-1}\right)(2m-1)^{-\frac{\varepsilon}{2}\ell}.\]
\end{lem}
\begin{proof}
Let $\tilde Y_i$ be the sub-2-complex of $\tilde Y$ consisting of faces labeled by the $i$ first abstract relators. Denote $P_i=\Pr\Cond{\tilde Y_i \textup{ is fillable by } G_\ell(m,d) }{Q_\ell}$. Apply lemma \ref{for sub 2-complexes} on $\tilde Y_i$, we have
\[P_i\leq \left(\frac{2m}{2m-1}\right)^i (2m-1)^{\sum_{j=1}^i\left( \eta_j + \left(d-1+\frac{\varepsilon}{2}\right)\ell \right)}.\]

Note that if $\tilde Y$ is fillable by $G_\ell(m,d)$ then its sub 2-complex $\tilde Y_i$ is fillable by the same group. So for any $1\leq i\leq k$,

\[\log_{2m-1}(P_k)\leq \log_{2m-1}(P_i) \leq \sum_{j=1}^i\left( \eta_j + \left(d-1+\frac{\varepsilon}{2}\right)\ell +\log_{2m-1}\left(\frac{2m}{2m-1}\right)\right).\]

Without loss of generality, suppose that $\alpha_1 \geq \alpha_2 \geq \dots \geq \alpha_k$. Note that $\log_{2m-1}(P_k)$ is negative and $\alpha_1\leq |\tilde Y|$, so $|\tilde Y|\log_{2m-1}(P_k)\leq \alpha_1\log_{2m-1}(P_k)$. By Abel's summation formula, with convention $\alpha_{k+1} = 0$,
\begin{align*}
    |\tilde Y|\log_{2m-1}(P_k)& \leq \alpha_1\log_{2m-1}(P_k) = \sum_{i=1}^k(\alpha_i-\alpha_{i+1})\log_{2m-1}(P_k) \\
    &\leq \sum_{i=1}^k(\alpha_i-\alpha_{i+1})\sum_{j=1}^i\left[\eta_i+\left(d-1+\frac{\varepsilon}{2}\right)\ell+\log_{2m-1}\left(\frac{2m}{2m-1}\right)\right]\\
    &= \sum_{i=1}^k\alpha_i\left[\eta_i+\left(d-1+\frac{\varepsilon}{2}\right)\ell+\log_{2m-1}\left(\frac{2m}{2m-1}\right)\right]\\
    &=\sum_{i=1}^k\alpha_i\eta_i + \left(\sum_{i=1}^k\alpha_i\right)\left[\left(d-1+\frac{\varepsilon}{2}\right)\ell+\log_{2m-1}\left(\frac{2m}{2m-1}\right)\right].
\end{align*}

Note that $\sum_{i=1}^k\alpha_i=|\tilde Y|$. By Lemma \ref{alphai etai 2} and the hypothesis of the current lemma, 
\[\sum_{i=1}^k\alpha_i\eta_i \leq |\tilde Y^{(1)}|+\Red(\tilde Y) < (1-d-\varepsilon)|\tilde Y|\ell.\]
Hence,

\begin{align*}
    |\tilde Y|\log_{2m-1}(P_k)& \leq (1-d-\varepsilon)|\tilde Y|\ell + |\tilde Y|\left[\left(d-1+\frac{\varepsilon}{2}\right)\ell+\log_{2m-1}\left(\frac{2m}{2m-1}\right)\right]\\
    &\leq |\tilde Y|\left[-\frac{\varepsilon}{2}\ell+\log_{2m-1}\left(\frac{2m}{2m-1}\right)\right].
\end{align*}
\end{proof}

\subsection{Proof of Theorem \ref{inequality for 2-complex}}\label{proof main theorem} Under the condition $Q_\ell := \{(2m-1)^{(d-\frac{\varepsilon}{4}\ell)} \leq |R_\ell| \leq (2m-1)^{(d+\frac{\varepsilon}{4}\ell)}\}$, the probability that there exists a van Kampen 2-complex of complexity $K$ of $G_\ell(m,d)$ satisfying the inverse inequality
\[|Y^{(1)}| + \Red(Y) < (1-d-\varepsilon)|Y|\ell\tag{$*$}\]
is bounded by
\[\sum_{\tilde Y \textup{ of complexity $K$, satisfying $(*)$}}\Pr\Cond{\tilde Y\textup{ is fillable by $G_\ell(m,d)$}}{Q_\ell}.\]

By Lemma \ref{number of abstract 2-complexes} and the face that there at most $K^{2K}$ ways to label a 2-complex with $K$ faces by abstract relators $\{1^\pm,\dots,K^\pm\}$, there are at most $\ell^{3K}\times K^{2K}$ terms in the sum. By Lemma \ref{key lemma}, every term is bounded by $\left(\frac{2m}{2m-1}\right)(2m-1)^{-\frac{\varepsilon}{2}\ell}$. So the sum is smaller than
\[\ell^{3K}K^{2K}\left(\frac{2m}{2m-1}\right)(2m-1)^{-\frac{\varepsilon}{2}\ell},\]
which converges to $0$ as $\ell\to\infty$.

By definition $\Pr(Q_\ell)\tendversl 1$, so the probability that there exists a van Kampen 2-complex of $G_\ell(m,d)$ of complexity $K$ satisfying $(*)$ converges to $0$ as $\ell$ goes to infinity. That is to say, a.a.s. every van Kampen diagram of $G_\ell(m,d)$ of complexity $K$ satisfies the inequality \[|Y^{(1)}| + \Red(Y) \geq (1-d-\varepsilon)|Y|\ell.\]\qed

\paragraph{Closed surfaces.}

Recall that a 2-complex $Y$ without isolated edges is called \textit{contractible} if there is an edge of $Y$ that is adjacent to one single face. If a 2-complex $Y$ is not contractible, then each of its edge is adjacent to at least $2$ faces, and we have $|Y^{(1)}| \leq \frac{1}{2}|Y|\ell$ where $\ell$ is the maximal boundary length of faces. As this contradicts the inequality of Theorem \ref{inequality for 2-complex} for \textit{any} density $d<1/2$, we have the following proposition.

\begin{prop} Let $(G_\ell(m,d))$ be a sequence of random groups with $m\geq2$ generators at density $d$. For any $d<1/2$ and $K>0$, a.a.s. every \textbf{reduced} van Kampen 2-complex of complexity $K$ of $G_\ell(m,d)$ can be contracted to a graph.
\end{prop}\quad

\section{Phase transition for the existence of van Kampen 2-complexes}\label{section existence of 2-complexes}

In this section, we work on the proof of Theorem \ref{existence of 2-complexes}. 

\paragraph{Motivation and a counterexample.}
Let $(G_\ell(m,d))$ be a sequence of random groups at density $d$. We are interested in the converse of Theorem \ref{inequality for 2-complex} without the reduction part: if a 2-complex $Y_\ell$ with bounded complexity satisfies the inequality
\[|Y_\ell^{(1)}|\geq (1-d+s)|Y_\ell|\ell\]
with some $s>0$, does there exist a face labeling by relators and an edge labeling by generators, so that $Y_\ell$ becomes a reduced van Kampen 2-complex of $G_\ell(m,d)$?

The motivation for this question comes from the well-known phase transition at density $d = \frac{\lambda}{2}$, mentioned in \cite[p. 274]{Gro93}: if $d<\frac{\lambda}{2}$ then a.a.s. $G_\ell(m,d)$ has the $C'(\lambda)$ small cancellation condition; while if $d>\frac{\lambda}{2}$ then a.a.s. $G_\ell(m,d)$ does not have $C'(\lambda)$. The first assertion is a simple application of Theorem \ref{inequality for diagrams}. For the second assertion, we need to show that a.a.s. there exists a van Kampen 2-complex $D$ of $G_\ell(m,d)$ with exactly 2 faces of boundary length $\ell$, sharing a common path of length at least $\lambda\ell$ (Figure \ref{lambda diagram}).

\begin{figure}[h]
    \centering 
    \begin{tikzpicture}
           \fill[gray!20] (30:1) circle (1) ;
            \fill[gray!20] (150:1) circle (1) ;
            \draw[very thick] (0,0) -- (0,1);
            \draw[very thick] (0,0) arc (-150:150:1);
            \draw[very thick] (0,1) arc (30:330:1);
            \node at (0.25,0.5) {$\lambda\ell$};
    \end{tikzpicture}
    \caption{A van Kampen diagram denying the $C'(\lambda)$ condition.}
    \label{lambda diagram}
\end{figure}
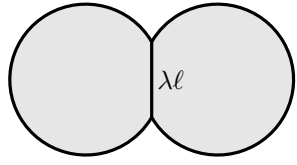 

The first detailed proof of such an existence is given in \cite[Theorem 2.1]{BNW20}, using an analog of the probabilistic pigeonhole principle. Another proof is given in \cite[Theorem 1.4]{Tsa21}. An intuitive explanation using the "dimension reasoning" is given in \cite{Oll05} p.30: The dimension of the set of couples $R_\ell\times R_ \ell$ is $2d\ell$. Sharing a common subword of length $L$ imposes $L$ equations, so the
"dimension" of the set of couples of relators sharing a common subword of length $\lambda\ell$ is $2d\ell-\lambda\ell$. If $d>\lambda/2$, then there will exist such a couple because the dimension will be positive. However, this argument is not true for \textit{any} 2-complex in general. Here is a counterexample: 

At density $d = 0.4$, let $(D_\ell)$ be a sequence of 2-complexes where $D_\ell$ is given in Figure \ref{figure counterexample}. The given inequality is satisfied because $|D_\ell^{(1)}| = 1.9 \ell > 1.8\ell = (1-d)|D_\ell|\ell$. However, the sub-diagram $D'_\ell$ gives $|D_\ell'^{(1)}| = 1.1\ell < 1.2\ell = (1-d)|D'_\ell|\ell$), which contradicts the isoperimetric inequality of Theorem \ref{inequality for 2-complex} and can not be a van Kampen diagram of $G_\ell(m,d)$.
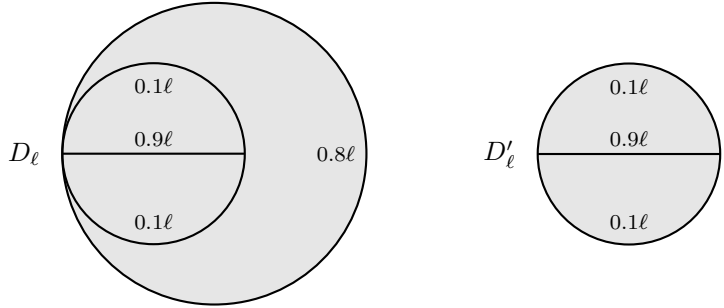
\begin{figure}[h]
\begin{center}
\begin{tikzpicture}
\fill [gray!20] circle (2);
\draw [thick] circle (2);
\draw [thick] (-0.8,0) circle (1.2);
\draw [thick] (-2,0) -- (0.4,0);
\node at (-0.8,0.2) {\footnotesize{$0.9\ell$}};
\node at (-0.8,-0.9) {\footnotesize{$0.1\ell$}};
\node at (-0.8,0.9) {\footnotesize{$0.1\ell$}};
\node at (1.6,-0) {\footnotesize{$0.8\ell$}};
\node at (-2.5,0) {$D_\ell$};
\end{tikzpicture}\qquad\qquad
\begin{tikzpicture}
\draw [color = white] circle (2);
\fill [gray!20] (-0.8,0) circle (1.2);
\draw [thick] (-0.8,0) circle (1.2);
\draw [thick] (-2,0) -- (0.4,0);
\node at (-0.8,0.2) {\footnotesize{$0.9\ell$}};
\node at (-0.8,-0.9) {\footnotesize{$0.1\ell$}};
\node at (-0.8,0.9) {\footnotesize{$0.1\ell$}};
\node at (-2.5,0) {$D'_\ell$};
\end{tikzpicture}
\end{center}
    \caption{A 2-complex that satisfies the isoperimetric inequality with a sub-2-complex that does not.}
    \label{figure counterexample}
\end{figure}

\subsection{Geometric form and Critical density}
Let us define the \textit{geometric form} of 2-complexes and the \textit{critical density} of a geometric form. To simplify the notations, for a 2-complex $Y = (V,E,F)$, we denote $\Edge(Y)$ as the set of geometric edges of $Y$ and $e$ instead of $\overline e$ for geometric edges.

\begin{defi}\label{geometric form} A \textbf{geometric form} of 2-complexes is a couple $(Y,\lambda)$ where $Y = (V,E,F)$ is a finite connected 2-complex without isolated edges, and $\lambda$ is a length labeled on edges defined by $\lambda: \Edge(Y) \to ]0,1], e \mapsto \lambda_e$, such that for every face $f$ of $Y$, the boundary length $|\partial f|$ is bounded by $1$.

A sequence of 2-complexes $(Y_\ell)$ is called \textup{of the geometric form} $(Y,\lambda)$ if $Y_\ell$ is obtained from $Y$ by dividing every edge $e$ of $Y$ into $\flr{\lambda_e\ell}$\footnote{We can replace $\flr{\lambda_e\ell}$ by any function with $\lambda\ell+o(\ell)$ and slightly smaller than $\lambda\ell$. Note that the sum of edge lengths on every face boundary of $Y_\ell$ is at most $\ell$} edges of length $1$.
\end{defi}

A sequence of 2-complexes $(Y_\ell)$ is briefly said to be \textit{of the same geometric form} if the geometric form $(Y,\lambda)$ is not specified. Note that the boundary length of every face $f$ of $Y_\ell$ is at most $\ell$. If $Z$ is a sub-2-complex of $Y$, we denote $Z\leq Y$. By convention, if $(Z_\ell)$ is a sequence of 2-complexes of the geometric form $(Z,\lambda_{|Z})$, we have $Z_\ell\leq Y_\ell$ for any integer $\ell$.

\begin{defi}\label{def density of Y} Let $(Y,\lambda)$ be a geometric form of 2-complexes. The \textbf{density} of $Y$ is
\[\dens(Y) := \frac{\sum_{e\in \Edge(Y)} \lambda_e}{|Y|}.\]
The \textbf{critical density} of $Y$ is 
\[\dens_c(Y) := \min_{Z\leq Y}\{\dens(Z)\}.\]
\end{defi}\quad

The intuition of this definition can be found in Lemma \ref{density of Y}: the density of $Y$ is actually the density of all possible van Kampen 2-complexes that fill $Y_\ell$.

\begin{rem}\label{theorem implies corollary} Taking Definition \ref{def density of Y} and Definition \ref{geometric form} together, we have
\[\dens(Y) = \frac{\sum_{e\in \Edge(Y)} \lambda_e}{|Y|} = \lim_{\ell\to\infty}\frac{\sum_{e\in \Edge(Y)} \flr{\lambda_e\ell}}{|Y_\ell|\ell} = \lim_{\ell\to\infty}\frac{|Y_\ell^{(1)}|}{|Y_\ell|\ell}.\]
Hence, the condition $``\dens_c(Y)+d>1"$ is equivalent to the following statement: Given $s>0$, for $\ell$ large enough, every sub-2-complex $Z_\ell$ of $Y_\ell$ satisfies \[|Z_\ell^{(1)}|\geq (1-d+s)|Z_\ell|\ell.\]

This argument shows that the second assertion of Theorem \ref{existence of 2-complexes} is equivalent to Corollary
\ref{cor of theorem}.
\end{rem}

\paragraph{Proof of Theorem \ref{existence of 2-complexes} (i).} We will use Theorem \ref{inequality for 2-complex} without the reduction part. Let $(G_\ell(m,d))$ be a sequence of random groups with $m$ generators at density $d$. Recall that a 2-complex $Y_\ell$ is said to be \textit{fillable} by $G_\ell(m,d)$ if there exists a \textit{reduced} van Kampen 2-complex of $G_\ell(m,d)$ whose underlying 2-complex is $Y_\ell$.

Let $(Y,\lambda)$ be a geometric form of 2-complexes with $\dens_c Y +d <1$. Let $(Y_\ell)$ be a sequence of 2-complexes of the geometric form $(Y,\lambda)$. We shall prove that a.a.s. the 2-complex $Y_\ell$ is \textit{not} fillable by the random group $G_\ell(m,d)$. By the definition of critical density, there exists a sub-2-complex $Z\leq Y$ satisfying $\dens Z +d <1$. Let $(Z_\ell)$ be the sequence of 2-complexes of the geometric form $(Z,\lambda_{|Z})$. We shall prove that a.a.s. $Z_\ell$ is not fillable by $G_\ell(m,d)$.

Let $\varepsilon >0$ such that $\dens Z = 1-d-3\varepsilon$. By definition,
\[\lim_{\ell\to\infty}\frac{|Z_\ell^{(1)}|}{|Z_\ell|\ell} = 1-d-3\varepsilon,\]
so for $\ell$ large enough,
\[|Z_\ell^{(1)}|\leq (1-d-2\varepsilon)|Z_\ell|\ell<(1-d-\varepsilon)|Z_\ell|\ell.\]

The complexity of $Z_\ell$ is $K=\max\left\{|Z|,|Z^{(1)}|,\max\{\frac{1}{\lambda_e} \,|\, e\in \Edge(Z)\}\right\}$, independent of $\ell$. By Theorem \ref{inequality for 2-complex} with $\varepsilon$ and $K$ given above, a.a.s. every van Kampen 2-complex $Z_\ell$ of $G_\ell(m,d)$ of complexity $K$ should satisfy 
\[|Z_\ell^{(1)}|\geq(1-d-\varepsilon)|Z_\ell|\ell.\]

Hence, a.a.s. the given 2-complex $Z_\ell$ is not fillable by $G_\ell(m,d)$, which implies that a.a.s. $Y_\ell$ is not fillable by $G_\ell(m,d)$. \qed

\subsection{The multidimensional intersection formula for random subsets}

To prove the second assertion of Theorem \ref{existence of 2-complexes}, we need the \textit{multidimensional intersection formula} for random subsets with density, introduced in \cite[Section 3]{Tsa21}.

Recall that $B_\ell$ is the set of cyclically reduced words of $X_m^\pm= \{x_1^\pm,\dots,x_m^\pm\}$ of length at most $\ell$, and that $|B_\ell| = (2m-1)^{\ell+o(\ell)}$. Let $k\geq 1$ be an integer. Denote $B_\ell^{(k)}$ as the set of $k$-tuples of pairwise distinct relators $(r_1,\dots,r_k)$ in $B_\ell$. Such a notation can be used for any set or any random set.

Note that $|B_\ell^{(k)}| = (2m-1)^{k\ell+o(\ell)}$. Recall that a sequence of fixed subsets $(\mathcal{Y}_\ell)$ of the sequence $(B_\ell^{(k)})$ is called \textit{densable with density} $\alpha\in \{-\infty\}\cup[0,1]$ if the sequence of real numbers $(\log_{|B_\ell^{(k)}|}|\mathcal{Y}_\ell|)$ converges to $\alpha$ (see \cite[p.272]{Gro93} and \cite[Definition 1.5]{Tsa21}). That is to say, $|\mathcal{Y}_\ell| = (2m-1)^{\alpha k\ell + o(\ell)}$.

\begin{defi}[Self-intersection partition, {\cite[Definition 3.4]{Tsa21}}] Let $(\mathcal{Y}_\ell)$ be a sequence of fixed subsets of the sequence $(B_\ell^{(k)})$. Let $0\leq i \leq k$ be an integer. The $i$-th self-intersection of $\mathcal{Y}_\ell$ is
\[S_{i,\ell} := \{(x,y)\in \mathcal{Y}_\ell^2\,|\, |x\cap y| = i\}\]
where $|x\cap y|$ is the number of common elements between the sets $x = (r_1,\dots,r_k)$ and $y = (r_1',\dots,r_k')$.
\end{defi}

The family of subsets $\{S_{i,\ell} \,| \, 0\leq i \leq k\}$ is a partition of $\mathcal{Y}_\ell^2$, called the \textit{self-intersection partition} of $\mathcal{Y}_\ell$. Note that $(S_{i,\ell})_{\ell\in\mathbb{N}}$ is a sequence of subsets of the sequence $\left((B_\ell^{(k)})^2\right)_{\ell\in \mathbb{N}}$, with density smaller than $\dens_{\left((B_\ell^{(k)})^2\right)}\left(\mathcal{Y}_\ell^2\right)=\dens_{(B_\ell^{(k)})}(\mathcal{Y}_\ell)$.

\begin{defi}[$d$-small self-intersection condition, {\cite[Definition 3.5]{Tsa21}}] Let $(\mathcal{Y}_\ell)$ be a sequence of fixed subsets of $\left(B_\ell^{(k)}\right)$ with density $\alpha$. Let $S_{i,\ell}$ with $0\leq i \leq k$ be its self-intersection partition. Let $d > 1-\alpha$. We say that $(\mathcal{Y}_\ell)$ satisfies the $d$-small self-intersection condition if, for every $1\leq i \leq k-1$,
\begin{equation*}
    \dens_{\left((B_\ell^{(k)})^2\right)}\left(S_{i,\ell}\right) < \alpha  - (1-d) \times \frac{i}{2k}.
\end{equation*}
\end{defi}

\begin{thm}[Multidimensional intersection formula, {\cite[Theorem 3.6]{Tsa21}}]\label{multidim intersection} Let $(R_\ell)$ be a sequence of permutation invariant random subsets of $(B_\ell)$ of density $d$. Let $(\mathcal{Y}_\ell)$ be a sequence of fixed subsets of $\left(B_\ell^{(k)}\right)$ of density $\alpha>1-d$. If $(\mathcal{Y}_\ell)$ satisfies the $d$-small self intersection condition, then the sequence of random subsets $(\mathcal{Y}_\ell \cap R_\ell^{(k)})$ is densable with density $\alpha+d-1$.

In particular, a.a.s. the random subset $\mathcal{Y}_\ell \cap R_\ell^{(k)}$ of $B_\ell^{(k)}$ is not empty.
\end{thm}

\subsection{Proof of Theorem \ref{existence of 2-complexes} (ii)} Let $(Y_\ell)$ be a sequence of 2-complexes of the same geometric form $(Y,\lambda)$ with $k$ faces. In the following, we denote $\mathcal{Y}_\ell$ as the set of pairwise distinct relators in $B_\ell$ that fills $Y_\ell$, which is a subset of $B_\ell^{(k)}$.

Let $(G_\ell(m,d))$ be a sequence of random groups at density $d$, defined by $G_\ell(m,d) = \langle X_m|R_\ell \rangle$ where $(R_\ell)$ is a sequence of random subsets with density $d$. The intersection $\mathcal{Y}_\ell \cap R_\ell^{(k)}$ is hence the set of $k$-tuples of pairwise distinct relators in $R_\ell$ that fills $Y_\ell$. We want to prove that this intersection is not empty, so that $Y_\ell$ is fillable by $G_\ell(m,d)$. According to Theorem \ref{multidim intersection}, it remains to prove that if $\dens_c Y>1-d$, then the sequence $(\mathcal{Y}_\ell)$ is densable and satisfies the $d$-small self intersection condition.\\

We will prove in Lemma \ref{density of Y} that $(\mathcal{Y}_\ell)$ is densable with density exactly $\dens(Y)$, and in Lemma \ref{d-small self intersection} that it satisfies the $d$-small self intersection condition.

\begin{lem}\label{number of all fillings} Let $\overline{\mathcal{Y}_\ell}$ be the set of $k$-tuples of relators in $B_\ell$ that fills $Y_\ell$, not necessarily pairwise distinct. If $Y_\ell$ is fillable by $B_\ell$, then
\[\dens_{(B_\ell^{k})}(\overline{\mathcal{Y}_\ell}) = \dens Y.\]
\end{lem}
\begin{proof} We shall estimate the number $|\overline{ \mathcal{Y}_\ell}|$ by counting the number of labelings on edges of $Y_\ell$ that produce van Kampen 2-complexes with respect to all possible relators $B_\ell$. 

We start by filling edges in the neighborhoods of vertices that are originally vertices of the geometric form $Y$ (before dividing). Consider the set of oriented edges of $Y_\ell$ starting at some vertex that is originally a vertex of $Y$ before dividing. A \textit{vertex labeling} is a labeling on these edges by $X_m^\pm$ that does not produce any reducible pair of edges on face boundaries: for every pair of different edges $e_1, e_2$ starting at the same vertex, if they are labeled by the same generator $x\in X_m^\pm$, then the path $e_1\moinsun e_2$ is not cyclically part of any face boundary loop. Since the 2-complex $Y_\ell$ is fillable, the set of vertex labelings is not empty. Denote $C\geq 1$ as the number of vertex labelings of $Y_\ell$.

As $m\geq 2$ and $\flr{\lambda_e\ell}\geq 3$ for $\ell$ large enough, if there exists a vertex labeling, then the other edges of $Y_\ell$ can be completed as a van Kampen 2-complex of $B_\ell$, and the number $C$ depends only on the geometric form $Y$.

To label the remaining $\flr{\lambda\ell}-2$ edges on the arc divided from the edge $e\in \Edge(Y)$, there are $2m-1$ choices for the first $\flr{\lambda\ell}-3$ edges, and $2m-2$ or $2m-1$ choices for the last edge. So

\[C\prod_{e\in \Edge(Y)}(2m-1)^{\flr{\lambda_e\ell}-3}(2m-2) \leq |\overline{\mathcal{Y}_\ell}| \leq C\prod_{e\in \Edge(Y)}(2m-1)^{\flr{\lambda_e\ell}-2}.\]
Recall that $k = |Y| = |Y_\ell|$ and that $|B_\ell^k| = (2m-1)^{k\ell+o(\ell)}$. We have
\[\dens_{(B_\ell^{k})}(\overline{\mathcal{Y}_\ell}) =\frac{\sum_{e\in \Edge(Y)}\lambda_e}{|Y|} = \dens Y.\]
\end{proof}\quad

\begin{lem}\label{density of Y} If $\dens_c Y >1/2$ and $Y_\ell$ is fillable by $B_\ell$, then $(\mathcal{Y}_\ell)$ is densable in $(B_\ell^{(k)})$ and \[\dens_{(B_\ell^{(k)})}(\mathcal{Y}_\ell) = \dens Y.\]
\end{lem}
\begin{proof} Suppose that $|Y| \geq 2$. The case $|Y|=1$ is trivial. Let $Z$ be a sub-2-complex of $Y$ with exactly two faces $f_1,f_2$. As $\dens Z \geq \dens_c Y > \frac{1}{2}$, by Definition \ref{def density of Y}, we have
\[\sum_{e\in \Edge(Z)}\lambda_e > \frac{1}{2}|Z| = 1 \geq |\partial f_1|.\]

Let $\overline{\mathcal{Y}^Z_\ell}$ be the set of fillings of $Y_\ell$ by $B_\ell$ such that the two faces of $Z$ are filled by the same relator. By the same arguments of the previous lemma,

\[|\overline{\mathcal{Y}^Z_\ell}|\leq C(2m-1)^{|\partial f_1|}\prod_{e\in \Edge(Y)\backslash\Edge(Z)}(2m-1)^{\flr{\lambda_e\ell}-2},\]
so 
\begin{align*}\dens_{(B_\ell^{k})}(\overline{\mathcal{Y}^Z_\ell}) & \leq \frac{1}{|Y|}\left[\sum_{e\in \Edge(Y)}\lambda_e  + \left(|\partial f_1| -  \sum_{e\in\Edge(Z)}\lambda_e\right)\right]\\
& < \frac{\sum_{e\in \Edge(Y)}\lambda_e}{|Y|}\\
& =\dens Y = \dens_{(B_\ell^{k})}(\overline{\mathcal{Y}_\ell}).
\end{align*}
Knowing that 
\[\mathcal{Y}_\ell = \left.\overline{\mathcal{Y}_\ell} \Big{\backslash} \bigcup_{Z<Y,|Z|=2}\overline{\mathcal{Y}^Z_\ell}\right.,\]
we have
\[|\overline{\mathcal{Y}_\ell}| - \sum_{Z<Y,|Z|=2} |\overline{\mathcal{Y}^Z_\ell}| \leq|\mathcal{Y}_\ell| \leq|\overline{\mathcal{Y}_\ell}|.\]
There are $\binom{|Y|}{2}$ terms in the sum, in every term we have $\dens_{(B_\ell^{k})}(\overline{\mathcal{Y}^Z_\ell}) < \dens_{(B_\ell^{k})}(\overline{\mathcal{Y}_\ell})$, so (cf. \cite[Proposition 2.7 and Proposition 2.8]{Tsa21})
\[\dens_{(B_\ell^k)}(\mathcal{Y}_\ell) = \dens_{(B_\ell^k)}(\overline{\mathcal{Y}_\ell)}.\]

Together with Lemma \ref{number of all fillings}, we have $\dens_{(B_\ell^k)}(\mathcal{Y}_\ell) = \dens Y$. As $\dens_{(B_\ell^k)}(B_\ell^{(k)}) = 1$, we get
\[\dens_{(B_\ell^{(k)})}(\mathcal{Y}_\ell) = \dens Y.\]
\end{proof}\quad

\begin{lem}\label{d-small self intersection} Suppose that $\dens_c Y > 1-d$. Let $S_{i,\ell}$ be the $i$-th self intersection of the set $\mathcal{Y}_\ell$. We have 
\[\dens_{\left((B_\ell^{(k)})^2\right)}\left(S_{i,\ell}\right) < \dens Y  - (1-d) \times \frac{i}{2k}.\]
\end{lem}
\begin{proof} Let $Z$, $W$ be two sub-2-complexes of $Y$ with $|Z|=|W|=i<k=|Y|$. Let $(Z_\ell),(W_\ell)$ be the corresponding sequences of 2-complexes of the geometric forms $Z$ and $W$ respectively. Denote $S_\ell(Z,W)$ as the set of pairs of pairwise distinct fillings $((r_1,\dots,r_k),(r_1'\dots,r_k'))$ of $Y_\ell$ by all possible relators $B_\ell$ such that, the $i$ relators in the first filling $(r_1,\dots,r_k)$ corresponding to $Z_\ell$ are identical to the $i$ relators in the second filling $(r_1',\dots,r_k')$ corresponding to $W_\ell$, and that the remaining $2k-2i$ relators are pairwise different, not repeating the relators in $Z_\ell$ and $W_\ell$.

Let us estimate the cardinality $|S_\ell(Z,W)|$. Firstly, fill the $k$-tuple $(r_1,\dots,r_k)$ so the $i$ relators in the next $k$-tuple $(r_1',\dots,r_k')$ corresponding to the sub 2-complex $W_\ell$ is determined. There are at most $i!$ choices for ordering these $i$ relators. To fill the remaining $k-i$ relators in $(r_1',\dots,r_k')$, by the same arguments of Lemma \ref{number of all fillings}, we get
\[|S_\ell(Z,W)| \leq |\mathcal{Y}_\ell|\times i! \times C\prod_{e\in \Edge(Y)\backslash\Edge(W)}(2m-1)^{\flr{\lambda_e\ell}-2}.\]

Recall that the density of $Y$ is defined by $\frac{1}{|Y|}\left(\sum_{e\in Edge(Y)}\lambda_e\right)$, and that $\dens W \geq \dens_c Y >1-d$ by Definition \ref{def density of Y}. Together with the hypothesis $\dens_c Y >1-d$, we have
\begin{align*} \dens_{\left((B_\ell^{(k)})^2\right)}\left(S_\ell(Z,W)\right) & \leq \frac{1}{2k} \left(\sum_{e\in Edge(Y)} \lambda_e + \sum_{e\in \Edge(Y)\backslash\Edge(W)} \lambda_e\right)\\
& = \frac{1}{2k}\left(2\sum_{e\in Edge(Y)}\lambda_e - \sum_{e\in \Edge(W)}\lambda_e\right)\\
& = \dens Y - \frac{i}{2k}\dens W\\
& < \dens Y  - \frac{i}{2k}(1-d).
\end{align*}

Note that \[\displaystyle{S_{i,\ell} = \bigcup_{Z<Y,W<Y,|Z|=|W|=i}S_\ell(Z,W)}.\]

It is a union of $\binom{k}{i}^2$ subsets of densities strictly smaller than $\dens Y  - \frac{i}{2k}(1-d)$. According to \cite[Proposition 2.7]{Tsa21}, we have

\[\dens_{\left((B_\ell^{(k)})^2\right)}\left(S_{i,\ell}\right) < \dens Y - \frac{i}{2k}(1-d).\]

\end{proof}
This completes the proof of Theorem \ref{existence of 2-complexes}.\\

\section{Phase transitions for small cancellation conditions}

Let us recall small cancellation notions in \cite[p.240]{LS77}. A \textit{piece} with respect to a set of relators is a cyclic sub-word that appears at least twice. A group presentation satisfies the $C'(\lambda)$ small cancellation condition for some $0<\lambda<1$ if the length of a piece is at most $\lambda$ times the length of any relator that it appears. It satisfies the $C(p)$ small cancellation condition for some integer $p\geq 2$ if no relator is a product of fewer than $p$ pieces. 

\paragraph{The $C'(\lambda)$ condition.} Let $(G_\ell(m,d))$ be a sequence of random groups at density $d$. It is known that there is a phase transition at density $d = \lambda/2$ for the $C'(\lambda)$ condition (see \cite[p.274]{Gro93}, \cite[Theorem 2.1]{BNW20} and \cite[Theorem 1.4]{Tsa21}). We give here a much simpler proof using Theorem \ref{existence of 2-complexes}.

\begin{prop}\label{lambda small cancellation plus} Let $0<\lambda<1$. Let $(G_\ell(m,d))$ be a sequence of random groups at density $d$. There is a phase transition at density $d = \lambda/2$:
\begin{enumerate}[(i)]
    \item If $d<\lambda/2$, then a.a.s. $G_\ell(m,d)$ satisfies $C'(\lambda)$.
    \item If $d>\lambda/2$, then a.a.s. $G_\ell(m,d)$ does not satisfy $C'(\lambda)$.
\end{enumerate}
\end{prop}

\begin{proof}\quad
\begin{enumerate}[(i)]
    \item Let us prove by contradiction. Suppose that a.a.s. $G_\ell(m,d)$ does not satisfy $C'(\lambda)$. That is to say, a.a.s. there exists a piece $w$ that appears in relators $r_1, r_2$ with $|w| > \lambda |r_1|$. It is possible that $r_1=r_2$, but the piece should be at different positions.
    
    Construct a van Kampen diagram $D$ by gluing two combinatorial disks with one face, labeled respectively by $r_1$ and $r_2$, along with the paths where the piece $w$ appears (Figure \ref{fig small cancellation}). As $r_1\neq r_2$ or $r_1=r_2$ but the piece appears at different positions, we obtain a reduced van Kampen diagram. The diagram satisfies $|D^{(1)}|=|r_1|+|r_2|+|w|<\ell+\ell+\lambda\ell<(1-\lambda/2)|D|\ell$, which contradicts Theorem \ref{inequality for 2-complex}.
    
    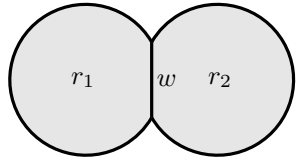
\begin{figure}[h]
    \begin{center}
    \begin{tikzpicture}
            \fill[gray!20] (30:1) circle (1) ;
            \fill[gray!20] (150:1) circle (1) ;
            \draw[very thick] (0,0) -- (0,1);
            \draw[very thick] (0,0) arc (-150:150:1);
            \draw[very thick] (0,1) arc (30:330:1);
            \node at (0.2,0.5) {$w$};
            \node at (0.9,0.5) {$r_2$};
            \node at (-0.9,0.5) {$r_1$};
    \end{tikzpicture}
    \end{center}
        \caption{A van Kampen 2-complex constructed from a $C'(\lambda)$ group.}
        \label{fig small cancellation}
    \end{figure}

    \item Consider a geometric form $Y$ with two faces sharing a common edge of length $\lambda$, the other two edges are of length $1-\lambda$ (Figure \ref{fig no small cancellation}). We have $\dens Y = \frac{2(1-\lambda)+\lambda}{2} > 1-d$, and every sub 2-complex with one face is with density $1>1-d$. So $\dens_c Y >1-d$.
    \begin{figure}[h]
    \begin{center}
    \begin{tikzpicture}
            \fill[gray!20] (30:1) circle (1) ;
            \fill[gray!20] (150:1) circle (1) ;
            \draw[very thick] (0,0) -- (0,1);
            \draw[very thick] (0,0) arc (-150:150:1);
            \draw[very thick] (0,1) arc (30:330:1);
            \node at (0.2,0.5) {$\lambda$};
            \node at (2.3,0.5) {$1-\lambda$};
            \node at (-2.4,0.5) {$1-\lambda$};
    \end{tikzpicture}
    \end{center}
        \caption{The geometric form for the $C'(\lambda)$ condition.}
        \label{fig no small cancellation}
    \end{figure}
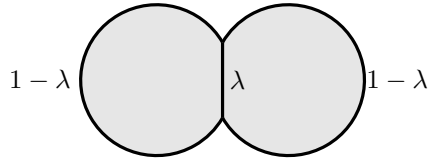

    Let $(Y_\ell)$ be a sequence of 2-complexes of the geometric form $Y$. By Theorem \ref{existence of 2-complexes}, a.a.s. $Y_\ell$ is fillable by $G_\ell(m,d)$, hence a.a.s. $G_\ell(m,d)$ does not satisfy $C'(\lambda)$.
\end{enumerate}

\end{proof}\quad

\paragraph{The $C(p)$ condition.} We shall prove by Theorem \ref{existence of 2-complexes} that for random groups with density, there is a phase transition at density $1/(p+1)$ for the $C(p)$ condition.

\begin{prop}\label{c of p} Let $p\geq 2$ be an integer. Let $(G_\ell(m,d))$ be a sequence of random groups at density $d$. There is a phase transition at density $d = 1/(p+1)$: 
\begin{enumerate}[(i)]
    \item If $d<1/(p+1)$, then a.a.s. $G_\ell(m,d)$ satisfies $C(p)$.
    \item If $d>1/(p+1)$, then a.a.s. $G_\ell(m,d)$ does not satisfy $C(p)$.
\end{enumerate}
\end{prop}

\begin{proof}\quad
\begin{enumerate}[(i)]
    \item Let us prove by contradiction. Suppose that a.a.s. $G_\ell(m,d)$ does not satisfy $C(p)$. That is to say, a.a.s. there exists a reduced van Kampen diagram $D$ with $(p+1)$ faces, one face is placed in the center, attached by the other $p$ faces on the whole boundary, and there is no other attachments (Figure \ref{fig small cancellation p}).
    
    \begin{figure}[h]
    \begin{center}
    \begin{tikzpicture}
    \foreach \x in {1,...,5}{
        \fill[gray!20] (72*\x:1.2) circle (0.7);
        \draw[very thick] (72*\x:1.2) circle (0.7);
        \begin{scope}[shift={(72*\x:1.2)}]
        \end{scope}
    }
    \fill[gray!20] (0,0) circle (1);
    \draw[very thick] (0,0) circle (1);
    \end{tikzpicture} 
    \end{center}
        \caption{A van Kampen 2-complex constructed from a $C(p)$ group.}
        \label{fig small cancellation p}
    \end{figure}
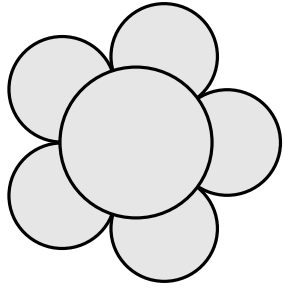

    We have $|D| = p+1$ and $|D^{(1)}|\leq p
    \ell$. Let $\varepsilon = \left(\frac{1}{p+1}-d\right)/2$, we have $|D^{(1)}|\geq (1-d-\varepsilon)|D|\ell$, which contradicts Theorem \ref{inequality for 2-complex} .
    
    \item Consider a geometric form with $p+1$ faces, one of the faces is placed in the center, having $p$ edges of length $1/p$, such that every edge is attached by another face with two edges of lengths $1/p$ and $1-1/p$. There are no other attachments (Figure \ref{fig no small cancellation p}).
    
    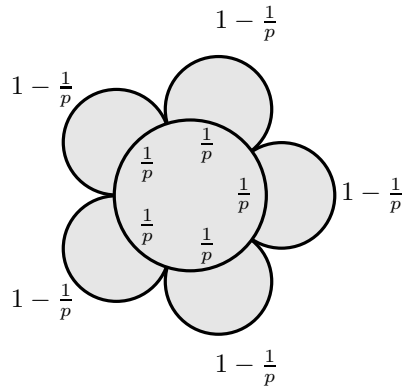
\begin{figure}[h]
    \begin{center}
    \begin{tikzpicture}
    \foreach \x in {1,...,5}{
        \fill[gray!20] (72*\x:1.2) circle (0.7);
        \draw[very thick] (72*\x:1.2) circle (0.7);
        \node at (72*\x:0.75) {$p$};
        \begin{scope}[shift={(72*\x:1.2)}]
        \end{scope}
    }
    \fill[gray!20] (0,0) circle (1);
    \draw[very thick] (0,0) circle (1);
    \foreach \x in {1,...,5}{
        \node at (72*\x:0.7) {$\frac{1}{p}$};
        \node at (72*\x:2.4) {$1-\frac{1}{p}$};
    }
    \end{tikzpicture}
    \end{center}
        \caption{The geometric form for the $C(p)$ condition.}
        \label{fig no small cancellation p}
    \end{figure}

    The density of $Y$ is $\frac{1+p(1-1/p)}{p+1} = p/(p+1)>1-d$. If $Z$ is a sub-2-complex of $Y$ not containing the center face, then $\dens Z=1$. If $Z$ contains the center face and $i\leq p$ other faces, then $\dens Z = \frac{1+i(1-1/p)}{i+1}>1-d$. So $\dens_c Y > 1-d$.

    Let $(Y_\ell)$ be a sequence of 2-complexes of the geometric form $Y$. By Theorem \ref{existence of 2-complexes}, a.a.s. $Y_\ell$ is fillable by $G_\ell(m,d)$, hence a.a.s. $G_\ell(m,d)$ does not satisfy $C(p)$.
\end{enumerate}
\end{proof}\quad

\paragraph{The $B(2p)$ condition.} The same argument holds for the $B(2p)$ condition, introduced in \cite[Definition 1.7]{OW11} by Y. Ollivier and D. Wise: half of a relator can not be the product of fewer than $p$ pieces. One can construct a geometric form with $p$ faces, one of the faces is in the center, with half of its boundary attached by the other $p$ faces, each with length $1/p$ (Figure \ref{fig no small cancellation B2p}). Its critical density is $\frac{p + \frac{1}{2}}{p+1}$, so a phase transition occurs at density $d = \frac{1}{2p+2}$.
    
\begin{prop}\label{b of 2p} Let $p\geq 1$ be an integer. Let $(G_\ell(m,d))$ be a sequence of random groups at density $d$. There is a phase transition at density $d = 1/(2p+2)$: 
\begin{enumerate}[(i)]
    \item If $d<1/(2p+2)$, then a.a.s. $G_\ell(m,d)$ satisfies $B(2p)$.
    \item If $d>1/(2p+2)$, then a.a.s. $G_\ell(m,d)$ does not satisfy $B(2p)$.\qed
\end{enumerate}
\end{prop}

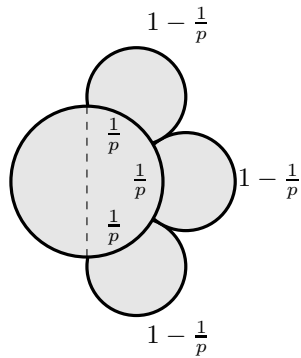
\begin{figure}[h]
    \begin{center}
    \begin{tikzpicture}
    \foreach \x in {5,...,7}{
        \fill[gray!20] (60*\x:1.3) circle (0.65);
        \draw[very thick] (60*\x:1.3) circle (0.65);
        \begin{scope}[shift={(60*\x:1.3)}]
        \end{scope}
    }
    \fill[gray!20] (0,0) circle (1);
    \draw[very thick] (0,0) circle (1);
    \draw[dashed] (0,-1)--(0,1);
    \foreach \x in {5,...,7}{
        \node at (60*\x:0.7) {$\frac{1}{p}$};
        \node at (60*\x:2.4) {$1-\frac{1}{p}$};
    }
    \end{tikzpicture}
    \end{center}
    \caption{The geometric form for the $B(2p)$ condition.}
    \label{fig no small cancellation B2p}
\end{figure}

\printbibliography[heading=bibintoc]
\end{document}